\documentclass[12pt, reqno]{amsart}
\usepackage{amsmath, amsthm, amscd, amsfonts, amssymb, graphicx}
\numberwithin{equation}{section}\usepackage{dsfont}
\input{amssym.def}
\input{amssym.tex}

\begin{document}

\title[Stability of  set-valued  generalized additive
 functional equation]
{Stability of  set-valued  generalized additive Cauchy functional equation}

\author[G. Lu]{Gang Lu$^*$}
\address{Gang Lyu \newline \indent Division of Foundational Teaching, Guangzhou College of Technology and Business, Guangzhou 510850, P.R. China}
\email{lvgang@ybu.edu.cn}\vskip 2mm
\author[M.Mu]{Mingzhu Mu}
\address{Mingzhu Mu\newline \indent Department of Mathematics, School of Science, ShenYang University of Technology, Shenyang 110870, P.R. China}
\email{28138389@qq.com}\vskip 2mm
\author[Y. Jin]{Yuanfeng Jin$^*$}
\address{Yuanfeng Jin \newline \indent Department of Mathematics,  Yanbian University, Yanji 133001, P.R. China}
\email{yfkim@ybu.edu.cn}\vskip 2mm

\author[C. Park]{Choonkil Park}
\address{Choonkil Park\newline \indent  Research Institute for Natural Sciences,
Hanyang University, Seoul 04763,   Korea}
\email{baak@hanyang.ac.kr}\vskip 2mm

\begin{abstract}

The main purpose of this paper is to determine the solution of generalized
convex set-valued mappings satisfying certain functional equation. Some conclusions of stability of set-valued functional equations are obtained.
\end{abstract}

\subjclass[2010]{Primary 54C60, 39B52, 47H04, 49J54.}

\keywords{Hyers-Ulam stability;  additive set-valued
 functional equation; closed and convex subset. \\$^*$Corresponding authors.}

\theoremstyle{definition}
  \newtheorem{df}{Definition}[section]
    \newtheorem{rk}[df]{Remark}
\theoremstyle{plain}
  \newtheorem{lem}[df]{Lemma}
  \newtheorem{thm}[df]{Theorem}
  \newtheorem{cor}[df]{Corollary}
    \newtheorem{prop}[df]{Proposition}
    \newtheorem{example}[df]{Example}
\setcounter{section}{0}

\maketitle

\baselineskip=14pt

\numberwithin{equation}{section}

\vskip .2in

\section{Introduction and preliminaries}

The goal of this paper is to characterize set-valued solutions of a generalized additive
Cauchy functional equation.  In 1992,   Rassias and Tabor \cite{RT} asked whether the functional
\begin{eqnarray*}
f(ax+by+c)=Af(x)+Bf(y)+C
\end{eqnarray*}
with $abAB\neq 0$ is stable in the sense of Hyers, Ulam, and Rassias. And the functional equation is called {\it  generalized additive
Cauchy functional equation}.

 Badea \cite{Ba} answered this question of Rassias and Tabor for the case when $c=C=0,a=A$ and $b=B$. As Gajda \cite{Ga} extended the result of  Badea  \cite{Ba} also modified the
result of the previous theorem by using a similar method.
  Jung \cite{JU} investigated the stability problem in complex Banach space for a generalized additive
Cauchy functional equation
\begin{eqnarray*}
f\left(x_0+\sum_{j=1}^{m}a_jx_j\right)=\sum_{i=1}^m b_i f\left(\sum_{j=1}^m a_{ij}x_j\right).
\end{eqnarray*}

An interesting question concerning set-valued mappings that satisfy some functional inclusions is the existence of solutions verifying certain properties. In the stability theory of functional equations, it is important to find solutions of set-valued mappings that are also solutions of certain functional equations.
 And more results were proved for set-valued mappings satisfying general additive Cauchy of the type
\begin{eqnarray}\label{eqn11}
f(ax+by+c)\subseteq Af(x)+Bf(y)+C
\end{eqnarray}
or
\begin{eqnarray}\label{eqn12}
Af(x)+Bf(y)\subseteq f(ax+by+c)+C
\end{eqnarray}
 where $f:X\rightarrow 2^Y, X$ and $Y$ are real vector spaces,
 $a,b,A,B\in \mathds{R}, c\in X$ and $C\in 2^Y$. By $2^Y$ we denote
 the collection of nonempty subsets of $Y$.
Nikodam and Popa \cite{NP1} considered the general solution of set-valued mappings satisfying linear  relation, which can be regarded as a generalization of the additive single-valued functional equation with $a=A,b=B$ and $C=c=0$. Lu {\it et al.} \cite{LP, POS} investigated the approximation of some set-valued functional
equations with $C=c=0$.
   Popa \cite{Po1} determined conditions for which a set-valued map that
verifies a relation of the form (\ref{eqn11}) admits a unique solution
$f:X\rightarrow Y$ that satisfies
\begin{eqnarray*}
f(ax+by+c)\subseteq Af(x)+Bf(y),
\end{eqnarray*}
for all $x,y \in X$. The case of equality of form (\ref{eqn11}) was also studied in  \cite{PV}. The similar results \cite{IP, NP2} were obtained in the case of (\ref{eqn12}).

Assume that $Y$ is a topological vector space satisfying the $T_0$-separation axiom. For real numbers $s,t$ and sets $A,B\subset Y$ we put $sA+tB:=\{y\in Y; y=sa+tb,a\in A,b\in B\}$. Suppose that the space $2^Y$ of all subsets of $Y$ is endowed with the Hausdorff topology (see  \cite{RH}). A set-valued function $F:X\rightarrow 2^Y$ is said to be additive if it satisfies the Cauchy functional equation $F(x_1+x_2)=F(x_1)+F(x_2), x_1,x_2\in X$,  where $X$ is a semigroup.

In this paper, we characterize set-valued solutions of the following functional equation
\begin{eqnarray}
f(ax+by+c)= AG(x)+BH(y)+C
\end{eqnarray} for all $x,y \in X$, where $a, b,A$ and $B$ are positive real numbers, $c$ and $C$ are  fixed bounded convex sets.

The family of all
closed and convex subsets of $Y$ will be denoted by $CC(Y)$, and the sets
of all real numbers, rational numbers and positive integers are denoted by $\mathbb{R}, \mathbb{Q}, \mathbb{N}$, respectively.

\begin{lem} {\rm \cite{NI}} \label{lem1.1}  Let $\lambda$ and $\mu$ be  real numbers.  If $A$ and $B$ are nonempty subsets of a real vector space
$X$,  then
\begin{eqnarray*}
\lambda(A+B)=\lambda A+\lambda B,  \\ (\lambda+\mu)A\subseteq
\lambda A+\mu B.\nonumber
\end{eqnarray*}
Moreover, if $A$ is  a convex set and $\lambda,\mu
\geq 0$, then we have
\begin{eqnarray*}(\lambda+\mu)A=\lambda A+\mu A.
\end{eqnarray*}
\end{lem}

\begin{lem} {\rm \cite{RA}} \label{lem1.2}
Let $A,B$ be subsets of $Y$ and assume that $B$ is closed and convex. If there exists a bounded and nonempty set $C\subset Y$ such that $A+C\subset B+C$, then $A\subset B$.
\end{lem}

The following lemmas  are rather known and can be easily verified. The proofs of them can be found in \cite{NI, KN}.

\begin{lem}\label{lem1.3}
If $\{A_n\}_{n\in \mathbf{N}}$ and $(B_n)_{n\in \mathbf{N}}$ are decreasing sequences of compact subsets of $Y$, then $\bigcap_{n\in \mathbb{N}}(A_n+B_n)=\bigcap_{n\in \mathbb{N}}A_n+\bigcap_{n\in \mathbb{N}}B_n.$
\end{lem}

\begin{lem}\label{lem1.4}
If $\{A_n\}_{n\in \mathbb{N}}$ is a decreasing sequence of compact subsets of  $Y$, then $A_n\rightarrow \bigcap_{n\in \mathbb{N}}A_n.$
\end{lem}

\begin{lem}\label{lem1.5}
If $A$ is a bounded subset of $Y$ and $(s_n)_{n\in \mathbb{N}}$ is a real sequence converging to an $s\in \mathds{R}$, then $s_nA\rightarrow sA.$
\end{lem}

\begin{lem}\label{lem1.6}
If $A_n\rightarrow A$ and $B_n\rightarrow B$, then $A_n+B_n\rightarrow A+B$.
\end{lem}

\begin{lem}\label{lem1.7}
If $A_n\rightarrow A$ and $A_n\rightarrow B$, then $cl A=cl B$.
\end{lem}

\section{Set-valued solution of the  Pexider functional equation}

In this section, we give the solution of the following set-valued functional equation
\begin{eqnarray}
F(ax+by+c)= AG(x)+BH(y)+C
\end{eqnarray} for all $x,y \in X$.

\begin{thm}
Assume that $(X,+)$ is a vector space and $Y$ is a $T_0$ topological vector space. If  set-valued  functions $F:X\rightarrow CC(Y), G:X\rightarrow CC(Y)$ and $H:X\rightarrow CC(Y)$ satisfy the functional  equation
\begin{eqnarray}\label{eqn1}
F(ax+by+c)=A G(x)+B H(y)+C
\end{eqnarray}
for all $x,y\in X$, where $a,b,A$ and $B$ are positive real numbers, $c$ and $C$ are fixed bounded convex sets. Then there exist an additive set-valued mapping  $F_0:X\rightarrow CC(Y)$ and sets $\alpha,\beta\in CC(Y)$ such that
\begin{eqnarray*}
F(x+c)= F_0\left(\frac{x}{a}\right)+K, \quad AG(x)=F_0(x)+A\alpha \quad and \quad BH(x)=F_0\left(\frac{bx}{a}\right)+B\beta
\end{eqnarray*}
for all $x\in X$.
\end{thm}

\begin{proof}
First, assume that $0 \in G(0)$ and $0\in H(0)$. Then, for all  $x, y\in X$, we have
\begin{eqnarray*}
\begin{split}
\;& AG(x)+BH(0)+C=F(2ax+c)=F\left(ax+b\frac{ax}{b}+c\right)\\
\;& =AG(x)+BH\left(\frac{ax}{b}\right)+C   \\
\;&\subset AG(x)+AG(0)+BH\left(\frac{ax}{b}\right)+C=AG(x)+F(ax+c)\\
\;&=AG(x)+AG(x)+BH(0)+C.
\end{split}\end{eqnarray*}
By Lemma \ref{lem1.2}, we get $G(2x)\subset 2G(x)$, which implies that the sequence
$(2^{-n}G(2^nx))_{n\in N}$ is decreasing. Put $F_0(x):=\bigcap_{n\in {\mathbb{N}}}2^{-n}G(2^nx), x\in X$. It is clear that $F_0(x)\in CC(Y)$ for all $x\in X$.
Similarly, we get $ H(2x)\subset 2H(x)$ for all $x\in X$. Then
\begin{eqnarray*}
A G(x)+B H(0)+C=F(a x+c)=AG(0)+BH\left(\frac{ax}{b}\right)+C.
\end{eqnarray*}
By Lemma \ref{lem1.3}, we obtain  that $A F_0(x)=B\bigcap 2^{-n}H\left(\frac{2^n ax}{b}\right)$ for all $x\in X$.
 From  (\ref{eqn1}), we get  $F(a2^n x+c)=A G(2^n x)+B H(0)+C, n \in \mathbb{N}$ and so
\begin{eqnarray*}
AF_0(x)=\bigcap_{n\in \mathds{N}} 2^{-n}F(a2^n x+c)
\end{eqnarray*}
 for all $x\in X$.

 Hence, using once more Lemma \ref{lem1.3}, we get
\begin{eqnarray*}
\begin{split}
AF_0(x_1+x_2)\;& =\bigcap_{n\in \mathbb{N}} 2^{-n}F(a2^nx_1+a2^n x_2+c)\\ \;&=\bigcap_{n\in \mathbb{N}}2^{-n}\left(A G\left(2^nx_1\right) +B H\left(\frac{2^na x_2}{b}\right)+C\right)\\
\;& =\bigcap_{n\in \mathbb{N}}2^{-n}A G\left(2^nx_1\right) +\bigcap_{n\in \mathbb{N}}2^{-n}B H\left(\frac{2^na x_2}{b}\right)+\bigcap_{n\in \mathbb{N}}2^{-n}C\\
\;&= AF_0(x_1)+AF_0(x_2), x_1,x_2\in X,
\end{split}
\end{eqnarray*}
which means that the set-valued function $F_0$ is additive.

Now observe that
\begin{eqnarray}\label{eqn3}
F(nax+c)+(n-1)AG(0)=F(ax+c)+(n-1)AG(x)
\end{eqnarray}
for all $x\in X$ and $n\in \mathds{N}$. Indeed, for $n=1$ the equality
is trivial. Assume that
it holds for a natural number $k$. Then, in virtue of (\ref{eqn1}), we obtain
 \begin{eqnarray*}
 \begin{split}
 \;& F((k+1)ax+c)+kA G(0)=A G(x)+B H\left(\frac{akx}{b}\right)+C +kA G(0)\\
 \;&=AG(x)+AG(0)+BH\left(\frac{akx}{b}\right)+C+kAG(0)-AG(0)\\
 \;& =AG(x)+F(akx+c)+(k-1)AG(0)\\
 \;&=AkG(x)+F(ax+c)
 \end{split}
 \end{eqnarray*}
 which proves that (\ref{eqn3}) holds for $n=k+1$. Thus, by induction, it holds for all $n\in \mathds{N}$. In particular, we have
 \begin{eqnarray*}
 F(2^n ax+c)+(2^n-1)AG(0)=F(ax+c)+(2^n-1)AG(x),
 \end{eqnarray*}
 and so
 \begin{eqnarray*}
 2^{-n}F(2^n ax+c)+(1-2^{-n})AG(0)=2^{-n}F(ax+c)+(1-2^{-n})AG(x)
 \end{eqnarray*} for all $x\in X$.
 By Lemma \ref{lem1.4}, $2^{-n}F(2^n ax+c)\rightarrow \bigcap_{n\in\mathds{N}}2^{-n}F(2^n ax+c)=F_0(x)$.

On the other hand, by Lemma \ref{lem1.5}, $1-2^{-n}G(0)\rightarrow G(0), 2^{-n}F(ax+c)\rightarrow
 \{0\}$ and $(1-2^{-n})G(x)\rightarrow G(x)$. Thus, using Lemmas \ref{lem1.6} and  \ref{lem1.7}, we get $cl [F_0(x)+AG(0)]=cl AG(x)$,
 whence $AG(x)=F_0(x)+AG(0)$ for all $x\in X$. Similarly, we can obtain $BH(x)=F_0\left(\frac{bx}{a}\right)+BH(0), x\in X$. Let $\alpha:=G(0)$ and $\beta:=H(0)$. Then $AG(x)=F_0(x)+A\alpha$ and $BH(x)=F_0\left(\frac{bx}{a}\right)+B\beta $ for all $x\in X$. Moreover $F(ax+c)=AG(x)+BH(0)+C= F_0(x)+AG(0)+BH(0)+C=F_0(x)+K, x\in X$, where $K=A\alpha+B\beta +F_0(0)+C$. This finishes our proof in the case that $0\in G(0)$ and $0\in H(0)$.

In the opposite case, fix arbitrarily points $a\in G(0)$ and $b\in H(0)$, and consider the set-valued mappings $F_1, G_1,H_1: X\rightarrow CC(Y)$ defined by $ G_1(x):=G(x)-a$ and $H_1:=H(x)-b, x\in X$. These set-valued mappings  satisfy the equation (\ref{eqn1}) and moreover $0\in G_1(0)$ and $0\in H_1(0)$. Therefore, by what we have discussed previously, we can get the same result. This completes the proof.
\end{proof}

\section*{Conclusion}

We have determined the solution of generalized
convex set-valued mappings satisfying certain functional equation. Some conclusions of stability of set-valued functional equations have been  obtained.

\bigskip

\section*{Declarations}

\bigskip

\noindent \textbf{Competing interests}\newline
\noindent The authors declare that they have no competing interests.

\bigskip

\noindent \textbf{Authors' contributions}\newline
\noindent The authors equally conceived of the study, participated in its
design and coordination, drafted the manuscript, participated in the
sequence alignment, and read and approved the final manuscript. \bigskip

\noindent {\bf Funding}
\newline \noindent
This work was supported by National Natural Science Foundation of China (No.  11761074), Project of Jilin Science and Technology Development for Leading Talent of Science and Technology Innovation in Middle and Young and Team Project(No.20200301053RQ),  and the scientific research project of Guangzhou College of Technology and Business in 2020 (No. KA202032).
\bigskip


\end{document}